\documentclass[11pt]{amsart}

\usepackage{latexsym}
\usepackage{amssymb}
\usepackage{amsmath}
\usepackage{color}

\newtheorem{theorem}{Theorem}[section]

\newtheorem{proposition}[theorem]{Proposition}

\newtheorem{example}[theorem]{Example}
\newtheorem{remark}[theorem]{Remark}

\newcommand{\cl}[1]{\mathcal{#1}}
\newcommand{\bb}[1]{\mathbb{#1}}

\newcommand\loc{\mathop{\rm loc}}
\newcommand\qq{\mathop{\rm q}}
\newcommand\aq{\mathop{\rm aq}}
\newcommand\qc{\mathop{\rm qc}}
\newcommand\xx{\mathop{\rm x}}

\usepackage{xcolor}

\begin{document}

\title{Products of synchronous games}


\author[]{L. Man\v{c}inska}
\address{QMATH, Department of Mathematical Sciences, University of Copenhagen}
\email{laura.mancinska@gmail.com}

\author[]{V. I. Paulsen}
\address{Institute for Quantum Computing and Department of Pure Mathematics, University of Waterloo, Waterloo, Canada}
\email{vpaulsen@uwaterloo.ca}

\author[]{I. G. Todorov}
\address{School of Mathematical Sciences, 
University of Delaware, 501 Ewing Hall,
Newark, DE 19716, USA}
\email{todorov@udel.edu}

\author[]{A. Winter}
\address{Grup d'Informaci\'o Qu\`antica, 
Departament de F\'isica,
Universitat Aut\`onoma de Barcelona,
08193 Bellaterra (Barcelona), Spain}
\email{andreas.winter@uab.cat}

\date{23 September 2021}

\begin{abstract} 
We show that the *-algebra of the product of two synchronous games is the tensor product of the corresponding *-algebras.
We prove that the product game has a perfect C*-strategy if and only if each of the 
individual games does, and that in this case the C*-algebra of the product game is 
*-isomorphic to the maximal C*-tensor product of the individual C*-algebras.
We provide examples of synchronous games whose synchronous values are 
strictly supermultiplicative.
\end{abstract}

\maketitle


\section{Introduction}\label{s_intro}

Non-local games have been extensively studied in Mathematics, Quantum Physics and Computer Science.
Among their numerous applications is 
the explicit demonstration of the presence of entanglement. 
More specifically, a probabilistic strategy for such a game is called {\it perfect}, 
if the probability that the players return a losing outcome is zero.
Many games, which can be shown to have no perfect classical probabilistic strategies, 
have been demonstrated to nevertheless 
possess perfect entanglement-assisted strategies.

A prominent family of non-local games is that of
{\it synchronous games}, introduced in \cite{psstw}.
Synchronous games play an important role in the recent (negative) resolution of 
Connes' Embedding Problem and an equivalent conjecture of Tsirelson \cite{jnvwy}. 
In particular, the authors of \cite{jnvwy} construct a synchronous game that has a perfect strategy 
in one mathematical model for entanglement description, 
but no perfect strategy in an alternate model, 
thus demonstrating that these two entanglement models are genuinely distinct.
We note that an important reformulation of the two problems was obtained by Kirchberg  
(see the monograph \cite{Pi2}); in fact, Kirchebrg's approach underlies many of the connections 
bridging operator algebras and non-local game theory (see \cite{heilbronn}). 

One of the reasons 
that synchronous games play a significant 
role in the resolution of the aforementioned problems, as well as more generally in 
quantum information theory, is that each 
synchronous game has a *-algebra associated with it
and the representation theory of this algebra captures the existence 
of perfect strategies for each of the various models used to describe entanglement \cite{hmps}.


Central concepts in non-local game theory are the various notions of 
a {\it value} of a game. Given a probabilistic strategy for a 
non-local game and a prior probability distribution on its inputs, 
one may compute the expected probability of winning. 
The different notions of a value of the game arise by considering the optimal expected 
winning probability over different, fixed, sets of strategies. 
Thus, non-local games have 
classical values as well as values defined using each of the mathematical models for entanglement. 
For synchronous games it is natural to also study these optimal winning probabilities
over various families of \emph{synchronous strategies}. This is the topic of the paper \cite{hmnpr}.  

One prominent question regarding non-local games is their behaviour under parallel repetition and, more generally, under taking products. We will show that the *-algebra affiliated with a product of synchronous games is the tensor product of the corresponding *-algebras. Thus, the existence/non-existence of perfect strategies is related to the behaviour of 
the various types of representations, after the tensor products of the individual *-algebras is taken. 
We show, in particular, that the product game has a perfect C*-strategy \cite{hmps} if and only if each of the individual game does so and that, in this case, 
the C*-algebra of the product game is canonically C*-isomorphic to the maximal C*-tensor product of the two individual C*-algebras. 

The behaviour of the classical value of the game under parallel repetition is relatively well understood \cite{raz} and there are simple examples of games known, 
for which the classical value is strictly supermultiplicative. The parallel repetition question has turned out 
to be much harder to understand in the case of quantum values, and the general question of whether the 
quantum value decays exponentially with the number of repetitions is still open.

 In \cite{hmnpr} the authors consider the parallel repetition question of synchronous values and show that they
 can be strictly supermultiplicative.  However, their example concerns the synchronous value of a non-synchronous game. We provide two examples of synchronous games whose synchronous values are strictly supermultiplicative, showing that this phenomenon can occur for the local, the quantum and the quantum commuting values alike.

\medskip

\noindent {\bf Acknowledgements. } 
The authors would like to thank Matthew Kennedy and 
Henry Yuen for useful discussions on the topic of this paper. 
L.M. was supported by Villum Fonden via the QMATH Centre of Excellence (Grant No. 10059) and Villum Young Investigator grant (No. 37532). V.I.P. was supported by NSERC grant 03784.
I.G.T. was supported by
NSF Grant 2115071. 
A.W. was supported by the Spanish MINECO (projects FIS2016-86681-P and PID2019-107609GB-I00/AEI/10.13039/501100011033), both with the support of FEDER funds, and by the Generalitat de Catalunya (project 2017-SGR-1127).
All the authors wish to thank the American Institute of Mathematics where this research originated.


\section{Preliminaries}\label{s_prel}

Let $X$, $Y$, $A$ and $B$ be finite sets, and recall that 
a POVM on a Hilbert space $H$ is a family $(E_i)_{i=1}^k$ of positive operators such that $\sum_{i=1}^k E_i = I$. 
A \emph{correlation} is a family $p = \{(p(a,b|x,y))_{a\in A, b\in B} : x\in X, y\in Y\}$, where 
$(p(a,b|x,y))_{a\in A, b\in B}$ is a probability distribution for every $(x,y)\in X\times Y$. 
A correlation $p$ is called
\begin{itemize}
\item[(i)] 
\emph{local} if it is the convex combination of 
correlations of the form $\{(p_1(a|x) p_2(b|y))_{a\in A,b\in B} : x\in X, y\in Y\}$ (notation: $\cl C_{\rm loc}$);
\item[(ii)] 
\emph{quantum} if there exist a finite dimensional Hilbert space $H$ (resp. $K$), a unit vector $\xi\in H\otimes K$ and 
POVM's $(E_{x,a})_{a\in A}$, $x\in X$ on $H$ (resp. $(F_{y,b})_{b\in B}$, $y\in Y$ on $K$), such that
$$p(a,b|x,y) = \langle (E_{x,a}\otimes F_{y,b})\xi,\xi\rangle$$ 
for all $x\in X, y\in Y, a\in A, b\in B$ (notation: $\cl C_{\rm q}$);
\item[(iii)] 
\emph{approximately quantum} if it belongs to the closure 
$\overline{\cl C_{\rm q}}$ of the set $\cl C_{\rm q}$ (notation: $\cl C_{\rm qa}$);
\item[(iv)] 
\emph{quantum commuting} if there exist a Hilbert space $H$, a unit vector $\xi\in H$ and
POVM's $(E_{x,a})_{a\in A}$, $x\in X$ (resp. $(F_{y,b})_{b\in B}$, $y\in Y$) on $H$, such that
$E_{x,a} F_{y,b} = F_{y,b}E_{x,a}$ for all $x,y,a,b$, and 
$$p(a,b|x,y) = \langle E_{x,a} F_{y,b}\xi,\xi\rangle, \ \ \ x\in X, y\in Y, a\in A, b\in B$$
for all $x\in X, y\in Y, a\in A, b\in B$ (notation: $\cl C_{\rm qc}$).
\end{itemize}
We note that a correlation of any of the types just defined is \emph{non-signalling} in the sense that 
\begin{equation}\label{eq_yy'}
\sum_{b\in B} p(a,b|x,y) = \sum_{b\in B} p(a,b|x,y'), \ \ x\in X, y,y'\in Y, a\in A,
\end{equation}
and 
\begin{equation}\label{eq_xx'}
\sum_{a\in A} p(a,b|x,y) = \sum_{a\in A} p(a,b|x',y), \ \ x,x'\in X,  y\in Y, b\in B.
\end{equation}
We denote the (convex) set of all non-signalling correlations by $\cl C_{\rm ns}$, and note the inclusions
$$\cl C_{\rm loc}\subseteq \cl C_{\rm q}\subseteq \cl C_{\rm qa}\subseteq \cl C_{\rm qc}\subseteq \cl C_{\rm ns}.$$

A \emph{non-local game}
is a tuple $\bb{G} = (X,Y,A,B,\lambda)$, where $X,Y,A,B$ are finite sets and 
$\lambda : X\times Y\times A\times B\to \{0,1\}$ is a map, called 
the \emph{rule function} of the game.
We interpret $X$ (resp. $Y$) as the set of questions posed by a Verifier to a player Alice (resp. Bob), and 
$A$ (resp. $B$) as the set of her (resp. his) possible answers. 
Alice and Bob play cooperatively against the Verifier; if a pair $(a,b)\in A\times B$ of answers is received given 
the pair $(x,y)\in X\times Y$ of questions, 
the players win (resp. lose) the round of the game if 
$\lambda(x,y,a,b) =1$ (resp. $\lambda(x,y,a,b) = 0$).

A \emph{probabilistic strategy} for the game $\bb{G}$ is a 
family $p = \{(p(a,b|x,y))_{a\in A,b\in B}$ $: x\in X, y\in Y\}$ of probability distributions; 
$p(a,b|x,y)$ is interpreted as the probability that the players respond with the pair $(a,b)$ of answers 
when they are asked the pair $(x,y)$ of questions. 
A probabilistic strategy $p$ for $\bb{G}$ is called \emph{non-signalling} if $p$ is a non-signalling correlation;
such a strategy being used by the players expresses the fact that they do not communicate
after the start of the game.
A non-signalling strategy $p$ for $\bb{G}$ is called \emph{perfect} if 
$$\lambda(x,y,a,b) = 0 \ \Longrightarrow \ p(a,b|x,y) = 0, \ \ \ x\in X, y\in Y, a\in A, b\in B.$$
For ${\rm x}\in \{{\rm loc}, {\rm q}, {\rm qa}, {\rm qc}, {\rm ns}\}$, let 
$\cl C_{\rm x}(\lambda)$ be the set of all elements of 
$\cl C_{\rm x}$ that are perfect strategies for the game $\bb{G}$ with rule function $\lambda$.

The game $\bb{G}$ is called \emph{synchronous} if $X = Y$, $A = B$, and 
$\lambda(x,x,a,b) = 0$ if $a\neq b$; in this case we write $\bb{G} = (X,A,\lambda)$.
The \emph{synchronicity game} has rule function 
$\lambda : X\times X\times A\times A\to \{0,1\}$ given by $\lambda(x,y,a,b) = 1-(1-\delta_{a,b})\delta_{x,y}$. 
For synchronous games, there are two further types of winning strategies which will be of interest to us. 
Namely, a synchronous game $\bb{G} = (X,A,\lambda)$ is said to have a 
\emph{perfect algebraic strategy} (see \cite{hmps}) 
if there exists a unital complex *-algebra
with a generating set 
$\{e_{x,a} : x\in X, a\in A\}$ satisfying the following relations:
$$e_{x,a} = e_{x,a}^* = e_{x,a}^2, \ \ x\in X, a\in A;$$
\begin{equation}\label{eq_cond}
\sum_{a\in A} e_{x,a} = 1, \ \ x\in X, \mbox{ and }
\end{equation}
$$\lambda(x,y,a,b) = 0 \Longrightarrow e_{x,a}e_{y,b} = 0.$$
If $\bb{G}$ has a perfect algebraic strategy then 
there exists a universal 
*-algebra $\cl A(\bb{G})$ satisfying the conditions (\ref{eq_cond})
(in the sense that any unital *-algebra satisfying (\ref{eq_cond}) is a canonical quotient of $\cl A(\bb{G})$)
and that $\cl A(\bb{G})$ is unique up to a unital *-isomorphism. 
We will reserve the notation $e_{x,a}$, $x\in X$, $a\in A$, for 
a family of generators of $\cl A(\bb{G})$ satisfying (\ref{eq_cond}). 

The game $\bb{G}$ is said to have a \emph{perfect} \emph{C*-strategy} if 
there exists a (universal) unital C*-algebra $C^*(\bb{G})$ with generating set  $\{e_{x,a} : x\in X, a\in A\}$ satisfying the conditions (\ref{eq_cond}). 
It is clear that if $\bb{G}$ has a perfect C*-strategy then it has a perfect algebraic strategy. 
The following observation is also clear.

\begin{remark}\label{r_harder}
{\rm 
For $i\in\{1,2\}$, let $\bb{G}_i$ be a synchronous game with rule function 
$\lambda_i : X_i \times X_i \times A_i\times A_i \to \{0,1\}$, $i = 1,2$. 
Suppose that $\lambda_1\leq \lambda_2$. 
If $\bb{G}_1$ has a perfect algebraic strategy (resp.~perfect $C^*$-strategy) then 
$\bb{G}_2$ has a perfect algebraic strategy (resp.~perfect $C^*$-strategy).}
\end{remark}

The synchronicity game with a question set $X$ and an answer set $A$ has a perfect C*-strategy, 
and its C*-algebra, which will be denoted by $\cl A_{X,A}$, is *-isomorphic, via Fourier transform, to the group C*-algebra 
$C^*(\bb{F}_{X,A})$, where 
$$\bb{F}\mbox{}_{X,A} = \underbrace{\bb{Z}\mbox{}_{|A|}\ast\cdots\ast \bb{Z}\mbox{}_{|A|}}_{|X| \ \mbox{times}}.$$
The following facts were established in \cite{heilbronn, kps, psstw}.

\begin{theorem}\label{th_cha}
Let $p\in \cl C_{\rm ns}$ be a synchronous correlation over the pair $(X,A)$ of finite sets.
\begin{itemize}
\item[(i)]
$p\in \cl C_{\rm qc}$ if and only if there exists a trace $\tau : \cl A_{X,A} \to \bb{C}$ such that 
\begin{equation}\label{eq_tra}
p(x,y,a,b) = \tau(e_{x,a}e_{y,b}), \ \ \ x,y\in X, a,b\in A;
\end{equation}

\item[(ii)]
$p\in \cl C_{\rm qa}$ if and only if there exists an amenable trace 
$\tau : \cl A_{X,A}\to \bb{C}$ such that 
(\ref{eq_tra}) is satisfied;

\item[(iii)]
$p\in \cl C_{\rm q}$ if and only if there exist a finite dimensional Hilbert space $H$, 
a *-representation $\pi : \cl A_{X,A} \to \cl B(H)$ and a trace $\tau_{\rm fin} : \cl B(H) \to \bb{C}$ 
such that 
\begin{equation}\label{eq_tra1}
p(x,y,a,b) = (\tau_{\rm fin}\circ \pi)(e_{x,a}e_{y,b}), \ \ \ x,y\in X, a,b\in A;
\end{equation}

\item[(iv)]
$p\in \cl C_{\rm loc}$ 
if and only if there exists a finite dimensional Hilbert space $H$, 
a *-representation $\pi : \cl A_{X,A} \to \cl B(H)$ with an abelian image, and a trace $\tau_{\rm fin} : \cl B(H) \to \bb{C}$ 
such that (\ref{eq_tra1}) is satisfied.
\end{itemize}
\end{theorem}


\section{Algebraic strategies of product games}

Let $\bb{G}_1 = (X_1,A_1,\lambda_1)$ and $\bb{G}_2 = (X_2,A_2,\lambda_2)$
be synchronous games. 
Their \emph{product} $\bb{G}_1\times \bb{G}_2$ is the game 
with pair of input sets $(X_1\times X_2, X_1\times X_2)$, 
pair of output sets $(A_1\times A_2,A_1\times A_2)$, and a rule 
function
$$\lambda_1\times\lambda_2 : X_1\times X_2 \times X_1\times X_2\times A_1\times A_2 \times A_1\times A_2 \to \{0,1\}$$
given by 
$$
(\hspace{-0.02cm}\lambda_1\hspace{-0.02cm}\times\hspace{-0.02cm}\lambda_2\hspace{-0.02cm})(\hspace{-0.02cm}
(\hspace{-0.02cm}x_1,x_2\hspace{-0.02cm}),
\hspace{-0.02cm}(\hspace{-0.02cm}y_1,y_2\hspace{-0.02cm}),
\hspace{-0.02cm}(\hspace{-0.02cm}a_1,a_2\hspace{-0.02cm}),
\hspace{-0.02cm}(\hspace{-0.02cm}b_1,b_2\hspace{-0.02cm})\hspace{-0.02cm}) 
\hspace{-0.02cm} = \hspace{-0.02cm}
\lambda_1(\hspace{-0.02cm}x_1\hspace{-0.02cm},\hspace{-0.02cm}y_1\hspace{-0.02cm},\hspace{-0.02cm}a_1\hspace{-0.02cm},\hspace{-0.02cm}b_1\hspace{-0.02cm}) 
\lambda_2(\hspace{-0.02cm}x_2\hspace{-0.02cm},\hspace{-0.02cm}y_2\hspace{-0.02cm},\hspace{-0.02cm}a_2\hspace{-0.02cm},\hspace{-0.02cm}b_2\hspace{-0.02cm})\hspace{-0.04cm}.$$

\medskip

\noindent {\bf Remarks. (i) }
If $\bb{G}_1$ and $\bb{G}_2$ are synchronous then so is $\bb{G}_1\times \bb{G}_2$; we thus write 
$\bb{G}_1\times \bb{G}_2 = (X_1\times X_2,A_1\times A_2,\lambda_1\times\lambda_2)$. 
A similar product can be defined for any (not necessarily synchronous) pair of non-local games, 
but in the sequel we will only be interested in the synchronous case.

\smallskip

{\bf (ii) } 
The product game $\bb{G}_1\times \bb{G}_2$ can be thought of as a game where Alice and Bob play the games $\bb{G}_1$ and $\bb{G}_2$ in parallel, receiving a question for each of them. 
In order to win the combined game, it is necessary and sufficient to win 
both $\bb{G}_1$ and $\bb{G}_2$.
Thus, $\bb{G}_1\times \bb{G}_2$ can be thought of as a conjunctive product of $\bb{G}_1$ and $\bb{G}_2$.
We note that there are also other natural ways to combine the games $\bb{G}_1$ and $\bb{G}_2$ (notably, 
by forming a disjunctive product), but they will not be considered in this paper.

We will denote by $\cl A_1\otimes\cl A_2$ the algebraic tensor product of *-algebras $\cl A_i$, $i = 1,2$. Recall that $\cl A_1\otimes\cl A_2$ is a *-algebra with an underlying vector space is the algebraic tensor product of the underlying spaces, multiplication given by 
$$(u_1\otimes u_2)(v_1\otimes v_2) := (u_1 v_1)\otimes (u_2 v_2),$$
and involution given by 
$$(u_1\otimes u_2)^* = u_1^* \otimes u_2^*.$$
If $\cl A_1$ and $\cl A_2$ are moreover C*-algebras, we write as customary $\cl A_1\otimes_{\max}\cl A_2$ for their \emph{maximal C*-tensor product}; thus, $\cl A_1\otimes_{\max}\cl A_2$ is the unique C*-algebra containing 
$\cl A_1$ and $\cl A_2$ as unital C*-subalgebras, generated by $\cl A_1$ and $\cl A_2$ as a C*-algebra, 
and satisfying the following universal property: 
if $H$ is a Hilbert space and $\pi_i : \cl A_i\to \cl B(H)$ is a *-representation, $i = 1,2$, such that $\pi_1(u_1)\pi_2(u_2) = \pi_2(u_2)\pi_1(u_1)$ for all $u_1\in \cl A_1$, $u_2\in \cl A_2$, then there exists a unique *-representation $\pi_1\otimes_{\max} \pi_2 : \cl A_1\otimes_{\max} \cl A_2\to \cl B(H)$ such that $(\pi_1\otimes_{\max} \pi_2)(u_i) = \pi_i(u_i)$, if $u_i\in \cl A_i$, $i = 1,2$.

\begin{theorem}\label{th_proiff}
Let $\bb{G}_1$ and $\bb{G}_2$ be synchronous games. 
\begin{itemize}
\item[(i)] The games $\bb{G}_1$ and $\bb{G}_2$ have perfect algebraic strategies if and only if 
$\bb{G}_1\times \bb{G}_2$ does.
In this case, $\cl A(\bb{G}_1\times \bb{G}_2) \cong \cl A(\bb{G}_1) \otimes \cl A(\bb{G}_2)$.

\item[(ii)] The games $\bb{G}_1$ and $\bb{G}_2$ have perfect C*-strategies if and only if $\bb{G}_1\times \bb{G}_2$ does. 
In this case, $C^*(\bb{G}_1\times \bb{G}_2) \cong C^*(\bb{G}_1)\otimes_{\max} C^*(\bb{G}_2)$.
\end{itemize}
\end{theorem}

\begin{proof}
Write 
$\bb{G}_1 = (X_1,A_1,\lambda_1)$, $\bb{G}_2 = (X_2,A_2,\lambda_2)$; 
set $\bb{G} = \bb{G}_1\times \bb{G}_2$ and $\lambda = \lambda_1\times \lambda_2$.

\smallskip

(i) Assume that $\bb{G}_1$ and $\bb{G}_2$ have perfect algebraic strategies.
For $(x_1,x_2)\in X_1\times X_2$ and $(a_1,a_2)\in A_1\times A_2$, set 
$g_{(x_1,x_2),(a_1,a_2)} = e_{x_1,a_1}\otimes e_{x_2,a_2}$; then conditions (\ref{eq_cond}) are clearly satisfied
for the family $\{g_{(x_1,x_2),(a_1,a_2)} : x_i\in X_i, a_i\in A_i, i = 1,2\}$. 
It follows that $\cl A(\bb{G}_1\times \bb{G}_2)$ exists and that there exists a surjective *-homomorphism 
$\pi : \cl A(\bb{G}_1\times \bb{G}_2) \to \cl A(\bb{G}_1)\otimes \cl A(\bb{G}_2)$
such that 
\begin{equation}\label{eq_pi}
\pi\left(e_{(x_1,x_2),(a_1,a_2)}\right) = e_{x_1,a_1} \otimes e_{x_2,a_2}, \ \ \ x_1\in X_1, x_2\in X_2, a_1\in A_1, a_2\in A_2.
\end{equation}

Conversely, suppose that $\bb{G}_1\times \bb{G}_2$ has a perfect algebraic strategy.
For $x_1\in X_1$, $x_2\in X_2$ and $a_1\in A_1$, let 
$$f_{x_1,x_2,a_1} = \sum_{b\in A_2} e_{(x_1,x_2),(a_1,b)}.$$
We have
\begin{equation}\label{eq_sumnew}
\sum_{a\in A_1} f_{x_1,x_2,a} = 1.
\end{equation}
Let $x_1, x_2, x_2' \in X$, $a_1\in A_1$, and suppose that $x_2\neq x_2'$. 
Since $\bb{G}_1$ is synchronous, 
\begin{equation}\label{eq_x1x20}
e_{(x_1,x_2),(a,b)} e_{(x_1,x_2'),(a_1,c)} = 0, \ \ \mbox{ for all } a,a_1\in A_1, b,c\in A_2, a\neq a_1.
\end{equation}
By (\ref{eq_sumnew}) and (\ref{eq_x1x20}), 
\begin{eqnarray*}
f_{x_1,x_2,a_1} f_{x_1,x_2',a_1}
& = & 
\left(1 - \sum_{a\in A_1, a\neq a_1} f_{x_1,x_2,a}\right) f_{x_1,x_2',a_1}\\ 
& = & 
f_{x_1,x_2',a_1} - \sum_{a\in A_1, a\neq a_1}\sum_{b,c\in A_2} e_{(x_1,x_2),(a,b)} e_{(x_1,x_2'),(a_1,c)}\\
& = & 
f_{x_1,x_2',a_1}.
\end{eqnarray*}
On the other hand, 
\begin{eqnarray}
f_{x_1,x_2,a_1} f_{x_1,x_2',a_1}
& = & 
f_{x_1,x_2,a_1} \left(1 - \sum_{a\in A_1, a\neq a_1} f_{x_1,x_2',a}\right)
\nonumber\\ 
& = & 
f_{x_1,x_2,a_1} - \sum_{a\in A_1, a\neq a_1}\sum_{b,c\in A_2} e_{(x_1,x_2),(a_1,c)} e_{(x_1,x_2'),(a,b)}\label{eq_product2}\\
& = & 
f_{x_1,x_2,a_1}.\nonumber
\end{eqnarray}
Thus, 
$f_{x_1,x_2,a_1} = f_{x_1,x_2',a_1}$; we set $f_{x_1,a_1} = f_{x_1,x_2,a_1}$ for any $x_2\in X$. From the calculation~(\ref{eq_product2}) it now follows that $f_{x_1,a_1}^2=f_{x_1,a_1}$. 
Clearly, we also have that $f_{x_1,a_1}^* = f_{x_1,a_1}$. 
Finally, if $\lambda_1(x_1,x_1',a_1,a_1') = 0$ then 
$$\lambda(x_1,x_2,x_1',x_2',a_1,a_2,a_1',a_2') = 0, \ \ x_2, x_2'\in X_2, a_2, a_2'\in A_2,$$
and hence, fixing some $x_2\in X_2$, we have 
$$ f_{x_1,a_1} f_{x_1',a_1'} = \sum_{b,c\in A_2} e_{(x_1,x_2),(a_1,b)} e_{(x_1',x_2),(a_1',c)} = 0.$$
It follows that the subalgebra $\cl A_1$ of $\cl A(\bb{G}_1\times \bb{G}_2)$, generated by the family
$\{f_{x_1,a_1} : x_1\in X_1, a_1\in A_1\}$, satisfies the conditions of a perfect algebraic strategy for $\bb{G}_1$.

Let the projections $f_{x_2,a_2}\in \cl A$, $x_2\in X_2$, $a_2\in A_2$, be defined similarly, 
and let $\cl A_2$ be the subalgebra of $\cl A(\bb{G}_1\times \bb{G}_2)$ generated by them. 
Let $\rho_i : \cl A(\bb{G}_i)\to \cl A_i$ be the quotient map satisfying 
$$\rho_i(e_{x_i,a_i}) = f_{x_i,a_i}, \ \ \ x_i\in X_i, a_i\in A_i,$$
which exists by the virtue of the universality of $\cl A(\bb{G}_i)$, $i = 1,2$. 
Then $\rho_1\otimes \rho_2 : \cl A(\bb{G}_1)\otimes \cl A(\bb{G}_2) \to \cl A_1\otimes \cl A_2$ is a *-homomorphism. 

For $x_i\in X_i$ and $a_i\in A_i$, $i = 1,2$, we have 
\begin{eqnarray*}
f_{x_1,a_1} f_{x_2,a_2}
& = & 
\left(\sum_{b\in A_2} e_{(x_1,x_2),(a_1,b)}\right) \left(\sum_{a\in A_1} e_{(x_1,x_2),(a,a_2)}\right)\\
& = & 
e_{(x_1,x_2),(a_1,a_2)} = f_{x_2,a_2} f_{x_1,a_1}.
\end{eqnarray*}
It follows that 
$s_1 s_2 = s_2 s_1$ for all $s_i\in \cl A_i$, $i = 1,2$; 
thus, by the universal property of the tensor product, the mapping 
$$m : \cl A_1\otimes \cl A_2 \to \cl A(G_1\times G_2), \ \ \ m(s_1\otimes s_2) = s_1 s_2,$$
is a *-homomorphism. 
It follows that the map 
$\rho := m \circ (\rho_1\otimes \rho_2) : \cl A(\bb{G}_1)\otimes \cl A(\bb{G}_2) \to  \cl A(\bb{G}_1\times \bb{G}_2)$
is a *-homomorphism; moreover, 
\begin{equation}\label{eq_rho}
\rho(e_{x_1,a_1} \otimes e_{x_2,a_2}) = e_{(x_1,x_2),(a_1,a_2)}, \ \ x_i\in X_i, a_i\in A_i, i = 1,2.
\end{equation}
Equations (\ref{eq_pi}) and (\ref{eq_rho}) show that 
$\pi\circ\rho$ agrees with the identity on the generators of $\cl A(\bb{G}_1)\otimes \cl A(\bb{G}_2)$, and hence on the whole 
algebra. It follows that $\pi$ is a *-isomorphism from 
$\cl A(\bb{G}_1\times \bb{G}_2)$ onto $\cl A(\bb{G}_1)\otimes \cl A(\bb{G}_2)$.

\smallskip

(ii)
If $C^*(\bb{G}_1)$ and $C^*(\bb{G}_2)$ exist then, by setting 
$g_{(x_1,x_2),(a_1,a_2)} = e_{x_1,a_1}\otimes e_{x_2,a_2}$ inside the maximal tensor product 
$C^*(\bb{G}_1)\otimes_{\max} C^*(\bb{G}_2)$, 
we obtain a family of operators that satisfy the relations defining the C*-algebra of $\bb{G}_1\times \bb{G}_2$; thus, $C^*(\bb{G}_1)\otimes_{\max} C^*(\bb{G}_2)$
implements a perfect C*-strategy for $\bb{G}_1\times \bb{G}_2$. 
Conversely, if $\bb{G}_1\times\bb{G}_2$ has a perfect C*-strategy 
then the closure $\tilde{\cl A}_1$ in $C^*(\bb{G}_1\times\bb{G}_2)$ of the *-algebra $\cl A_1$ defined in the proof of (i)
provides a perfect C*-strategy for $\bb{G}_1$. By symmetry, such exists for $\bb{G}_2$ as well. 

Now suppose that $\bb{G}_1$ and $\bb{G}_2$ have perfect C*-strategies. 
Then, by the universal properties of $C^*(\bb{G}_1\times \bb{G}_2)$ and 
$C^*(\bb{G}_1) \otimes_{\max} C^*(\bb{G}_2)$, the maps $\pi$ and $\rho$ defined in (i) extend to *-homomorphisms
$\tilde{\pi} : C^*(\bb{G}_1\times \bb{G}_2) \to C^*(\bb{G}_1) \otimes_{\max} C^*(\bb{G}_2)$ and 
$\tilde{\rho} :  C^*(\bb{G}_1) \otimes_{\max} C^*(\bb{G}_2) \to C^*(\bb{G}_1\times \bb{G}_2)$, which shows that 
$C^*(\bb{G}_1\times \bb{G}_2) \cong C^*(\bb{G}_1)\otimes_{\max} C^*(\bb{G}_2)$ canonically. 
\end{proof}

\noindent {\bf Remark. }
Let $\bb{G}_i = (X_i, A_i, \lambda_i)$, $i = 1,2$, be synchronous games and 
${\rm x}\in \{{\rm loc}, {\rm q}, {\rm qa}, {\rm qc}, {\rm ns}\}$. 
If $p_i\in \cl C_{\rm x}(\lambda_i)$, $i = 1,2$,
$$p_1\otimes p_2 : A_1\times A_2 \times A_1\times A_2 \times X_1\times X_2 \times X_1\times X_2  \to [0,1],$$ 
and
$$(p_1\otimes p_2)(a_1,a_2,b_1,b_2 | x_1,x_2,y_1,y_2) = 
p_1(a_1,b_1 | x_1,y_1) p_2(a_2,b_2 | x_2,y_2),$$
then
$p_1\otimes p_2 \in \cl C_{\rm x}(\lambda_1\times \lambda_2)$.
Indeed, 
it is straightforward that $p_1\otimes p_2 \in \cl C_{\rm ns}(\lambda_1\times\lambda_2)$. 
Suppose that $p_i\in \cl C_{\rm qc}$, $i = 1,2$, and write
$$p_i(a_i,b_i | x_i,y_i) = \langle E_{x_i,a_i}F_{y_i,b_i}\xi_i,\xi_i\rangle, \ \ \ i = 1,2,$$ 
as in the definition of quantum commuting correlations. 
Then
\begin{eqnarray*}
& & (p_1 \otimes p_2)(a_1,a_2,b_1,b_2 | x_1,x_2,y_1,y_2)\\
& = & \langle ((E_{x_1,a_1}\otimes E_{x_2,a_2})(F_{y_1,b_1} \otimes F_{y_2,b_2})(\xi_1\otimes \xi_2), (\xi_1\otimes \xi_2)\rangle,
\end{eqnarray*}
showing that $p_1\times p_2 \in \cl C_{\rm qc}$. 
The argument for $\xx = \qq$ and $\xx = \loc$ is similar, while the statement about 
${\rm qa}$ follows from the definition of $\cl C_{\rm qa}$ as the closure of $\cl C_{\qq}$. 

It is well-known, and easily seen that, 
in the converse direction, 
if $p \in \cl C_{\rm x}(\lambda_1\times \lambda_2)$ and $x_2,y_2\in X_2$, then the 
correlation
$$p_{x_2,y_2} : A_1\times A_1\times X_1\times X_1 \to [0,1],$$ 
given by
$$p_{x_2,y_2}(a_1,b_1|x_1,y_1) = \sum_{a_2,b_2\in A_2} p(a_1,a_2,b_1,b_2 | x_1,x_2,y_1,y_2),$$
is a perfect strategy for $\bb{G}_1$.
The argument from Theorem \ref{th_proiff} shows that, if the games $\bb{G}_1$ and $\bb{G}_2$ are synchronous, the correlation 
$p_{x_2,y_2}$ is independent of the choice of $x_2$ and $y_2$.

\section{Synchronous Values of Games}

A synchronous game $\bb{G} = (X,A,\lambda)$, equipped with a probability distribution $\pi$ on the 
set $X\times X$ of questions, will be called a \emph{synchronous game with density}. 
For each 
${\rm x}\in \{{\rm loc}, {\rm qa}, {\rm qc}, {\rm ns}\}$, we
define the \emph{synchronous ${\rm x}$-value} $\omega_{\rm x}^{\rm s}(\bb G,\pi)$ of $(\bb G, \pi)$ by setting 
\begin{equation}\label{eq_value}
\omega_{\rm x}^{\rm s}(\bb{G},\pi) = \sup\left\{\sum_{x\in X, y\in Y} \sum_{a\in A,b\in B} 
\pi(x,y) \lambda(x,y,a,b)p(a,b|x,y) : p\in \cl C_{\rm x}^{\rm s}\right\},
\end{equation}
where $\cl C_{\rm x}^{\rm s}$ denotes the set of all synchronous correlations over $(X,A)$ of class ${\rm x}$. 

Note that each term in the supremum is the expected probability of the players winning the game given that they use the conditional probability density $p$. So the synchronous {\rm x}-value represents the optimal winning probability over the corresponding set of synchronous correlations.

\begin{remark}\label{r_sv}
{\rm 
\begin{itemize}
\item[(i)]
Let ${\rm x}\in \{{\rm loc}, {\rm qa}, {\rm qc}, {\rm ns}\}$. Then 
$\omega_{\rm x}^{\rm s}(\bb{G},\pi) \leq 1$ and, since $\cl C_{\xx}^{\rm s}$ is closed, 
the supremum in the definition of $\omega_{\rm x}(\bb{G},\pi)$ is achieved. 
It follows that a synchronous game $G$ has a perfect ${\rm x}$-strategy if and only of $\omega_{\rm x}^{\rm s}(\bb{G}) = 1$.

\item[(ii)]
It was shown in \cite{kps} that $\cl C_{\rm qa}^{\rm s} = \overline{\cl C_{\rm q}^{\rm s}}$. Thus, 
if in the right hand side of (\ref{eq_value}) one employs correlations of the class $\cl C_{\rm q}^{\rm s}$, 
the corresponding supremum coincides with $\omega_{\rm qa}^{\rm s}(\bb{G},\pi)$. 
\end{itemize}
}
\end{remark}

\medskip

Given synchronous games with densities $(\bb{G}_i, \pi_i) = (X_i,A_i,\lambda_i,\pi_i)$, $i = 1,2$, equip 
the question set $(X_1\times X_2) \times (X_1\times X_2)$ of the game $\bb{G}_1\times \bb{G}_2$
with the probability distribution $\pi_1\times\pi_2$, given by 
$$(\pi_1\times\pi_2)\left((x_1,x_2),(y_1,y_2)\right) = \pi_1(x_1,y_1)\pi_2(x_2,y_2).$$ We write $(\bb G_1 \times \bb G_2, \pi_1 \times \pi_2)$ for this synchronous game with density.

\begin{proposition}\label{th_prv}
Let $(\bb{G}_i, \pi_i) = (X_i,A_i,\lambda_i,\pi_i)$, $i = 1,2$, be synchronous games with densities and 
${\rm x}\in \{{\rm loc}, {\rm qa}, {\rm qc}, {\rm ns}\}$. 
Then 
$$\omega_{\xx}^{\rm s}(\bb{G}_1,\pi_1)\omega_{\xx}^{\rm s}(\bb{G}_2,\pi_2)\leq 
\omega_{\xx}^{\rm s}(\bb{G}_1\times \bb{G}_2,\pi_1\times\pi_2).$$

\end{proposition}

\begin{proof}
If $p_i(a_i,b_i|x_i,y_i) \in \cl C_{\rm x}, \, i=1,2$, then applying Theorem~\ref{th_cha}, it is easily checked that
\[ p((a_1,a_2),(b_1,b_2)| (x_1,x_2), (y_1,y_2)) := p_1(a_1,b_1|x_1,y_1)p_2(a_2,b_2|x_2,y_2) \in \cl C_{\rm x}.\]
Consequently,
\begin{eqnarray*}
& & \omega_{\xx}^{\rm s}(\bb{G}_1,\pi_1)\omega_{\xx}^{\rm s}(\bb{G}_2,\pi_2)\\
& = & 
\prod_{i=1}^2 \sup\{\sum_{x_i,y_i,a_i,b_i} \pi_i(x_i,y_i) \lambda_i(x_i,y_i,a_i,b_i)
p(a_i,b_i | x_i,y_i) : p_i\in \cl C_{\rm x}^{\rm s}\}\\
& \leq & 
\sup\{\sum_{x_i,y_i,a_i,b_i, i =1,2} 
\pi_1(x_1,y_1) \pi_2(x_2,y_2) \lambda (x_1,x_2,y_1,y_2,a_1,a_2,b_1,b_2)\\
& & 
p(a_1,a_2,b_1,b_2 | x_1,x_2,y_1,y_2) : p\in \cl C_{\rm x}^{\rm s}\}\\
& = & 
\omega_{\xx}^{\rm s}(\bb{G}_1\times \bb{G}_2,\pi_1\times\pi_2).
\end{eqnarray*} 
\end{proof}

The inequality in Theorem \ref{th_prv} can be strict, even in the case of the 
classical value $\omega_{\rm loc}$. However, it is known that if $\omega_{\rm loc}(\bb G, \pi) <1$ then
\[ \omega_{\rm loc}(\bb G^n, \pi^n) \to 0,\]
where $\bb G^n$ represents the $n$-fold product of the game $\bb G$ with itself. 
In \cite{hmnpr} an example of a game is given for which the sequence 
of values $\omega_{\rm loc}^s(\bb G^n, \pi^n)$ is strictly increasing with limit 1.
Thus showing that this fundamental result fails utterly in the context of synchronous values.

However, the game in \cite{hmnpr} was not synchronous.
We provide two examples towards this end. The first is an example of a synchronous game with uniform distribution 
on the inputs for which the inequality is strict. 
The second example is of a {\it symmetric} synchronous game with uniform distribution on the 
inputs for which the inequality is strict.   

A rule function $\lambda$ is called {\it symmetric} if $\lambda(x,y,a,b) = \lambda(y,x,b,a)$.  
Note that, since synchronous correlations $p$ arise from traces, we have that 
\[ p(a,b|x,y) = \tau(e_{x,a}e_{y,b}) = \tau(e_{y,b}e_{x,a}) = p(b,a|y,x);\]
thus, they satisfy an extra transposition invariance.  
For this reason, arguably the most suitable 
context to study synchronous values is that of symmetric synchronous games 
with symmetric distributions $\pi$ on their inputs, that is, ones satisfying 
$\pi(x,y) = \pi(y,x)$ for all $x$ and $y$. 
Our examples illustrate that even in this more restrictive setting the synchronous value can be 
strictly supermultiplicative.  However, unlike in \cite{hmnpr}, the values are not increasing nor are we able to determine their asymptotic behaviour.

For a rule function $\lambda : X\times Y \times A\times B \to \{0,1\}$ of a 
game over the quadruple $(X,Y,A,B)$, and $(x,y)\in X\times Y$, let 
$$E_{x,y} = \{(a,b)\in A\times B : \lambda(x,y,a,b) = 1\}.$$
Clearly, specifying $\lambda$ is equivalent to specifying the family $\{E_{x,y} : x,y\}$ of subsets of $A\times B$. 
In synchronous games, we have $X = Y$, $A = B$, and 
$E_{x,x} \subseteq \{(a,a) : a\in A\}$, for every $x\in X$. 
Note that the local synchronous strategies of a synchronous game arise from functions
$f : X\to A$: a function $f : X\to A$ gives rise to the local no-signalling correlation 
$p_f = (p(a,b|x,y))$ given by $p_f(a,b|x,y) = 1$ if $a = f(x)$ and $b = f(y)$, and $p(a,b|x,y) = 1$ otherwise;
conversely, every local synchronous strategy is the convex combination of strategies of the form $p_f$
(this can be inferred, e.g., from \cite[Corollary 5.6]{psstw}). 

For a synchronous non-local game $\bb{G}$ with rule function $\lambda : X\times X\times A\times A\to \{0,1\}$ and probability distribution $\pi : X\times X\to [0,1]$, 
we write 
$$h_{\bb{G}} = \sum_{x,y\in X} \pi(x,y) \sum_{a,b\in A}\lambda(x,y,a,b)e_{x,a} e_{y,b},$$
viewed as an element of the C*-algebra $\cl A_{X,A}$.

\begin{example}\label{ex_1}
\rm
Let $\bb{G}$ be the game with $X = Y = A = B = [2] := \{1,2\}$ and rule function 
$\lambda$ specified by setting 
$$E_{1,1} = E_{2,2} = \{(1,1), (2,2)\}, E_{1,2} = \{(1,1)\}, \ E_{2,1} = \{(1,2)\}.$$
We claim that, if $[2]\times [2]$ is equipped with the uniform probability distribution $\pi$, then
\begin{equation}\label{eq_all1}
\omega^s_{\rm loc}(G, \pi) =\omega^s_{\rm q}(G, \pi) = \omega^s_{\rm qc}(G, \pi) = 3/4,
\end{equation}
while 
\begin{equation}\label{eq_all2}
\omega^s_{\rm loc}(\bb{G} \times \bb{G}, \pi \times \pi)
= 
\omega^s_{\rm q}(\bb{G} \times \bb{G}, \pi \times \pi)
= 
\omega^s_{\rm qc}(\bb{G} \times \bb{G}, \pi \times \pi)
=
10/16.
\end{equation}

Indeed, we have that
\[ h_{\bb{G}} = \frac{1}{4}(2I + e_{1,1} e_{2,1} + e_{2,1}e_{1,2}).\]
The synchronous values of 
$\bb{G}$ are obtained by computing the supremum of the quantities $\tau(h_{\bb{G}})$ 
over traces $\tau$ of $\cl A(\bb{G})$ of particular type.
Setting $p = e_{1,1}$ and $q = e_{2,1}$, we have
\[4h_{\bb{G}} = 2I + pq+ q(I-p) = 2I + q,\]
and (\ref{eq_all1}) follows.

Write $X_1 = X_2 = X$ and $A_1 = A_2 = A$ and let
$$\{E_{(x_1,x_2),(y_1,y_2)} : (x_1,y_1) \in X_1\times X_1, (x_2,y_2)  \in X_2\times X_2\}$$ 
be the family of subsets of $(A_1\times A_2)\times (A_1\times A_2)$ that determine $\lambda \times \lambda$. 
Note that 
$$E_{(x_1,x_2),(y_1,y_2)} = \{((a_1,a_2), (b_1,b_2)) : (a_1,b_1)\in E_{x_1,y_1} \mbox{ and } (a_2,b_2)\in E_{x_2,y_2}\}.$$
Consider the function $h : X_1\times X_2\to A_1\times A_2$, given by 
$$(1,1)\to (1,1), \ \ \ \ (1,2)\to (1,1), \ \ \ \ (2,1)\to (1,1), \ \ \ \ (2,2)\to (1,1).$$
For each $((x_1,x_2),(y_1,y_2))$ listed below, we indicate whether the value 
$(h\times h)((x_1,x_2),(y_1,y_2))$ belongs to the set $E_{(x_1,x_2),(y_1,y_2)}$:
\begin{itemize}
\item 
$E_{(1,1),(1,2)} = \{((1,1),(1,1)), ((2,1),(2,1))\}$ $\Longrightarrow$ Yes;

\item 
$E_{(1,1),(2,1)} = \{((1,1),(1,1)), ((1,2),(1,2))\}$ $\Longrightarrow$ Yes;

\item 
$E_{(1,1),(2,2)} = \{((1,1),(1,1))\}$ $\Longrightarrow$ Yes;

\item 
$E_{(1,2),(1,1)} = \{((1,1),(1,2)), ((2,1),(2,2))\}$ $\Longrightarrow$ No;

\item 
$E_{(1,2),(2,1)} = \{((1,1),(1,2))\}$ $\Longrightarrow$ No;

\item 
$E_{(1,2),(2,2)} = \{((1,1),(1,1)),((1,2),(1,2))\}$ $\Longrightarrow$ Yes;

\item 
$E_{(2,1),(1,2)} = \{((1,1),(2,1))\}$ $\Longrightarrow$ No;

\item 
$E_{(1,2),(2,1)} = \{((1,1),(1,2))\}$ $\Longrightarrow$ No;

\item 
$E_{(2,1),(2,2)} = \{((1,1),(1,1)),((2,1),(2,1))\}$ $\Longrightarrow$ Yes;

\item 
$E_{(2,2),(1,1)} = \{((1,1),(2,2))\}$ $\Longrightarrow$ No;

\item 
$E_{(2,2),(1,2)} = \{((1,1),(2,1)),((1,2),(2,2))\}$ $\Longrightarrow$ No;

\item 
$E_{(2,2),(2,1)} = \{((1,1),(1,1)),((2,1),(2,1))\}$ $\Longrightarrow$ Yes.
\end{itemize}
The remaining pairs $((x_1,x_2),(y_1,y_2))$ are of the form $((x_1,x_2),(x_1,x_2))$; since the game $\lambda$ is 
synchronous, we have that 
$(h\times h)((x_1,x_2),(x_1,x_2))$ meets $E_{(x_1,x_2),(x_1,x_2)}$. This shows that 
$$\omega_{\rm loc}(\bb G \times \bb G, \pi \times \pi) \geq \frac{10}{16}.$$ 

On the other hand, let $h_{\bb{G}\times \bb{G}}$ be the element of $\cl A(\bb{G}\times \bb{G})$ with 
\begin{multline}
16 h_{\bb{G}\times \bb{G}} = 4 \cdot I + e_{11,11}\big(e_{12,11} +  e_{22,11} + e_{21,11} \big) + \\
e_{11,12} \big(e_{21,12} + e_{12,11} \big) + e_{11,21}e_{21,11}+\\
e_{11,22} \big( e_{22,11} + e_{21,12} \big) + \\
e_{12,21} \big( e_{21,12} + e_{22,11} \big) + e_{12,22}e_{22,12}  + \\
e_{21,11}e_{22,11} + e_{21,21}e_{22,21} + e_{21,12}e_{22,11}  +
+ e_{21,22} e_{22,21},
\end{multline}
and note that 
$$\omega^s_{\rm qc}(\bb{G}\times\bb{G},\pi\times\pi) = 
\sup\{\tau(h_{\bb{G}\times\bb{G}}) : \tau \mbox{ a trace on } \cl A(\bb{G}\times\bb{G})\}.$$
For any trace $\tau$, if $p,q \ge 0$ and $r \ge q$, then 
$0 \le \tau(pq) \le \tau(pr)$. We thus see that
\begin{eqnarray*}
16 \tau(h_{\bb{G} \times \bb{G}}) 
& \le & 
4 + \tau(3 e_{11,11}) + \tau(3 e_{11,12}) + \tau(3 e_{11,22}) + \tau(2e_{12,21})\\
& + & \tau(e_{12,22}) + \tau(e_{21,11}+e_{21,21} + e_{21,12} + e_{21,22} )\\
& \leq & 
4 + 3 +2+1 =10,
\end{eqnarray*}
and (\ref{eq_all2}) is proved.
\end{example}

The game, exhibited above, shows that strict inequality may take place in Proposition \ref{th_prv} even for the smallest non-trivial games. 
At the expense of increasing the number of inputs by one, the game can be made symmetric 
-- this is achieved in the next example, which also exhibits a separation between the 
local and the quantum synchronous value.

\begin{example}\label{ex_2}
\rm 
Let $X = [3]$, $A = [2]$, equip $X\times X$ with the uniform probability distribution, 
and let $\pi : X\times X\times A\times A \to \{0,1\}$ be the rule 
function determined by the sets 
\begin{itemize}
\item $E_{1,2} = E_{2,1} = E_{2,3} = E_{3,2} = \{(1,2)\}$;

\item $E_{1,3} = E_{3,1} = \{(2,1)\}$;

\item $E_{1,1} = E_{2,2} = E_{3,3} = \{(1,1), (2,2)\}$.
\end{itemize}
The eight possibilities for functions $g : X\to A$ are:
\begin{itemize}
\item 
$g_1\times g_1 \times g_1 : (1,2,3) \to (1,1,1)$, yielding $R_{g_1} = 3$;
\item 
$g_2\times g_2 \times g_2 : (1,2,3) \to (1,1,2)$, yielding $R_{g_2} = 5$;
\item 
$g_3\times g_3 \times g_3 : (1,2,3) \to (1,2,1)$, yielding $R_{g_3} = 5$;
\item 
$g_4\times g_4 \times g_4 : (1,2,3) \to (1,2,2)$, yielding $R_{g_4} = 5$;
\item 
$g_5\times g_5 \times g_5 : (1,2,3) \to (2,1,1)$, yielding $R_{g_5} = 5$;
\item 
$g_6\times g_6 \times g_6 : (1,2,3) \to (2,1,2)$, yielding $R_{g_6} = 5$;
\item 
$g_7\times g_7 \times g_7 : (1,2,3) \to (2,2,1)$, yielding $R_{g_7} = 5$;
\item 
$g_8\times g_8 \times g_8 : (1,2,3) \to (2,2,2)$, yielding $R_{g_8} = 3$.
\end{itemize}
Thus, $\omega_{\rm loc}(\bb{G},\pi) = \frac{5}{9}$.

Letting as before $X_1 = X_2 = X$ and $A_1 = A_2 = A$,
consider the function $\eta : X_1\times X_2\to A_1\times A_2$, given by 
$$
\eta(x_1,x_2) = 
\begin{cases}
(1,1) & \text{if } (x_1,x_2)\in \{(1,1), (1,3), (3,1), (3,3)\}\\
(1,2) & \text{if } (x_1,x_2)\in \{(1,2), (3,2)\}\\
(2,1) & \text{if } (x_1,x_2)\in \{(2,1), (2,3)\}\\
(2,2) & \text{if } (x_1,x_2) = (2,2).\\
\end{cases}
$$
Among the inputs $(x_1,x_2,y_1,y_2)$ of the product game $\mu\times \mu$ with 
$x_1\leq x_2$, $y_1\leq y_2$, and $(x_1,y_1)\neq (x_2,y_2)$, we single out the following
\begin{eqnarray*}
& & 
E_{(1,1),(1,2)} = \{((1,1),(1,2)), ((2,1),(2,2))\},\\
& & 
E_{(1,1),(2,1)} = \{((1,1),(2,1)), ((1,2),(2,2))\},\\
& & 
E_{(1,1),(2,2)} = \{((1,1),(2,2))\},\\
& & 
E_{(1,2),(2,2)} = \{((1,1),(2,1)), ((1,2),(2,2))\},\\
& & 
E_{(2,1),(2,2)} = \{((1,1),(1,2)), ((2,1),(2,2))\},\\
& & 
E_{(1,3),(2,3)} = \{((1,1),(2,1)), ((1,2),(2,2))\},\\
& & 
E_{(3,1),(3,2)} = \{((1,1),(1,2)), ((2,1),(2,2))\},
\end{eqnarray*}
which all meet the graph of the function $\eta\times \eta$. 
Among the inputs $(x_1,x_2,$ $y_1,y_2)$ of the product game $\bb{G}\times \bb{G}$ with 
$x_1\leq x_2$, $y_1\geq y_2$, and $(x_1,y_1)\neq (x_2,y_2)$, we single out 
\begin{eqnarray*}
& & 
E_{(2,3),(2,2)} = \{((1,1),(1,2)), ((2,1),(2,2))\},\\
& & 
E_{(3,3),(3,2)} = \{((1,1),(2,1)), ((2,1),(2,2))\},
\end{eqnarray*}
which also meet the graph of $\eta\times \eta$.
By symmetry, we thus have that 
$$\omega_{\rm loc}(\bb G \times \bb G, \pi \times \pi) \geq \frac{27}{81} > \frac{25}{81} = \omega_{\rm loc}(\bb G, \pi)^2.$$
\end{example}

We finish this note by computing the
quantum value of the game $\bb{G}$ from Example \ref{ex_2}, implying that it is strictly 
greater than the local (classical) value. 


\begin{theorem}\label{p_ifclos}
Let $(\bb{G},\pi)$ be the non-local game from Example \ref{ex_2}. 
Then
$$\omega^s_{\rm qc}(\bb{G},\pi) = 
\omega^s_{\rm q}(\bb{G},\pi) = 7/12 > 5/9 = \omega^s_{\rm loc}(\bb{G},\pi).$$
\end{theorem}

\begin{proof}
By \cite{russell}, the set $\cl C^s_{\rm q}$ in our (three input, two output) case 
is closed. Thus, the supremum (\ref{eq_value}) 
used to compute $\omega_{\rm q}(\bb{G},\pi)$
is achieved. 
We write for brevity $e_{xa}$ in the place of $e_{x,a}$.
Let $\tau$ be a trace on a finite dimensional
representation of $\cl A_{X,A}$ 
for which the value $\omega_{\rm q}^{s}(\bb{G},\pi)$ is attained. 
By splitting the representation into irreducible blocks, we can assume that all projections $e_{xa}$ are contained in $M_n$ for some $n\in \bb{N}$. We note that
\begin{multline} 
9 h_{\bb{G}} = 3 I + e_{11}e_{22} + e_{21}e_{12} + e_{21}e_{32} + e_{31} e_{22} + e_{12} e_{31} + e_{32}e_{11}. 
\end{multline}
Setting $p = e_{11}$, $q = e_{21}$, and 
$r = e_{31}$, so that 
$e_{12} = I-p, e_{22} = I-q, e_{32} = I-r$ we have
$$
9 h_{\bb{G}} = 3 I + p(1-q) + q(1-p) + q(1-r) + r(1-q) + (1-p)r + (1-r)p.$$
By \cite{dpp}, 
$p$ commutes with $q+r$, $q$ commutes with $p+r$  and $r$ commutes with $p+q$.  This implies that $p+q+r$ commutes with $p$, $q$ and $r$.
Thus, 
$$
9 \tau(h_{\bb{G}}) = 3 + \tau(2p(1-q-r) + 2r + 2q(1-r)).$$
By \cite{krs}, 
there exists $t \in \bb R^+$ such that 
\begin{equation}\label{eq_p+q+r}
p+q+r = tI_n. 
\end{equation}
Write 
\[ p = \begin{pmatrix} I_k & 0\\0 & 0 \end{pmatrix}, \mbox{ so that } q+r = \begin{pmatrix} C & 0 \\ 0 & D \end{pmatrix},\]
and note that (\ref{eq_p+q+r}) yields
\[ C= (t-1) I_k \ \mbox{ and } \ D= tI_{n-k}.\]
Writing
$$q= \begin{pmatrix} A & X \\X^* & B \end{pmatrix},$$
we have 
$$r= \begin{pmatrix} (t-1)I_k - A & -X \\ - X^* & tI_{n-k} - B \end{pmatrix}.$$
The identities $q^2=q$ and $r^2= r$ yield
\[ A^2+XX^* = A,  \ \ \ ((t-1)I_k - A)^2 +XX^* = (t-1)I_k - A\]
and
\[ B^2+X^*X = B, \ \ \ 
(tI_{n-k}-B)^2 +X^*X = tI_{n-k} - B.\]
Thus, 
\[ ((t-1)I_k -A)^2 + A - A^2 = (t-1)I_k - A  \implies A = \frac{t-1}{2} I_k,\]
and 
\[ (tI_{n-k} - B)^2 +B - B^2 = tI_{n-k} -B \implies B = \frac{t}{2} I_{n-k}.\]
Therefore
\[ q = \begin{pmatrix} \frac{t-1}{2}I_k & X \\ X^* & \frac{t}{2} I_{n-k} \end{pmatrix}\]
and using again the fact that $q$ is an idempotent, we see that 
\[ XX^* = \left(\frac{t-1}{2} - 
\frac{t-1}{2})^2 \right) I_k
= \frac{4t - t^2 -3}{4} I_k\]
and
\[ X^*X = \frac{2t-t^2}{4} I_{n-k}.\]
Since $\tau(X^*X) = \tau(XX^*)$, we have
\[k(4t-t^2-3) =(2t - t^2)(n-k) \implies k(6t-2t^2-3) = (2t - t^2) n.\]
This shows that
\[ k/n= \frac{2t- t^2}{6t-2t^2-3}.\]

By \cite{krs}, the following values of $t$ are feasible:
\[ t= 0,1,2,3, 3/2.\]
This leads to
\[ k/n = 0,1,0,1, 1/2.\]
The case where $t=3/2$ and $k/n= 1/2$ can be realised in $M_2$ and we have that
\[ p = \begin{pmatrix} 1 & 0 \\ 0 & 0 \end{pmatrix}, q= \begin{pmatrix} 1/4 & \sqrt{3}/4 \\ \sqrt{3}/4 & 3/4 \end{pmatrix},  r= \begin{pmatrix} 1/4 & -\sqrt{3}/4 \\ -\sqrt{3}/4 & 3/4 \end{pmatrix}.\]
In this case  
\[\tau(h_{\bb{G}}) = 1/9\tau( 3I_2 + 2p(I -(tI -p) + 2r+ 2q(1-r))= 7/12.\] 

If $t=0$, then $q$ has a negative entry, so can be discarded.

If $t=1$, then again $q$ is not a projection.

If $t=2$, then $p=0$ and
\[q= \begin{pmatrix} 1/2 & 1/2\\1/2 & 1/2 \end{pmatrix}, r = \begin{pmatrix} 1/2 & -1/2 \\ - 1/2 & 1/2 \end{pmatrix}= 1-q,\]
and in this case
\[ \tau(h_G) =  1/9 \big({\rm tr}( 3I_2+ 2 r + 2 q^2)\big) = 5/9 < 7/12.\]

Finally, if $t=3$, then $p=q=r=I$ and so
\[ \tau(h_{\bb{G}}) = 1/9(3 +2(-1) +2 + 0) = 1/3.\]
So we get that the value occurs when $t=3/2$; note that we have also shown that, 
in addition, it can be achieved by matrices in $M_2$.

We turn to the quantum commuting value. 
Let $\tau : \cl A_{X,A}\to \bb{C}$ be a tracial state and $\pi_{\tau} : \cl A_{X,A}\to \cl B(H)$ be the GNS representation of $\cl A_{X,A}$, associated with $\tau$.
We let $\cl A = \pi_{\tau}(\cl A_{X,A})''$ and extend $\tau$ to a normal trace on $\cl A$ in the canonical fashion. 
Let $\cl Z$ be the centre of $\cl A$; up to a normal *-isomorphism, we have that $\cl Z = L^{\infty}(Z,\mu)$ for a suitable probability space $(Z,\mu)$. 
In the sequel, we use the terminology of \cite[Chapter 14]{kr}. 
Using \cite[Theorem 14.2.2]{kr}, we write 
$$H = \int_Z H_z d\mu(z) \ \ \mbox{ and } \ \ \cl A = \int_Z \cl A(z) d\mu(z)$$ in their corresponding direct integral decompositions. Note that $\cl A(z)\subseteq \cl B(H_z)$ is a factor for $\mu$-almost all $z\in Z$ (see \cite[Theorem 14.2.2]{kr}). For an element $T\in \cl A$, we write $T = \int_Z T(z) d\mu(z)$ for its direct integral decomposition. 

Suppressing notation, we write $p$, $q$ and $r$ for their images under $\pi_{\tau}$. By the arguments from the first part of the proof, $p + q + r \in \cl Z$; thus, there exists a measurable function $t : Z\to \bb{R}_+$, such that 
$$p(z) + q(z) + r(z) = t(z) I_{H_z}, \ \ \mbox{ for $\mu$-almost all } z\in Z.$$
We continue to use the notation from the first part of the proof. 
Thus, $C$ and $D$ are the (unique) operators in $\cl A$, determined by the requirement
$$q + r = pCp + p^{\perp} D p^{\perp},$$
while $A, B$ and $X$ are determined by 
$$q = pAp + p^{\perp} B p^{\perp} + 
pXp^{\perp} + p^{\perp} X^* p.$$
The previous arguments now imply 
\[\tau(p(z))(6t(z) - 2t(z)^2-3) = (2t(z) - t(z)^2).\]
Since the function $z\to t(z)$ is measurable (and takes finitely many values), the arguments above now show that there is a partition  
$Z = Z_{3/2}\cup Z_2 \cup Z_3$, where 
$$Z_{\alpha} = \{z\in Z : t(z) = \alpha\}.$$
Since $\mu$ is a probability measure, we have that $\tau(h_{\bb{G}})\leq 7/12$.
The proof is complete. 
\end{proof}


\end{document}